\documentclass[a4paper,reqno,11pt]{amsart}

\usepackage{amssymb,upref}
\usepackage[mathcal]{euscript}
\newcommand{\bydef}{:=}

\newcommand{\cA}{\mathcal{A}}%algebras
\newcommand{\calA}{\mathcal{A}}
\newcommand{\cB}{{\mathcal B}}
\newcommand{\calB}{{\mathcal B}}
\newcommand{\cC}{\mathcal{C}}
\newcommand{\cD}{{\mathcal D}}

\newcommand{\calC}{\mathcal{C}}
\newcommand{\calD}{\mathcal{D}}

\newcommand{\calJ}{\mathcal{J}}

\newcommand{\calL}{\mathcal{L}}

\newcommand{\calT}{\mathcal{T}}

\newcommand{\calW}{\mathcal{W}}

\newcommand{\calH}{{\mathcal H}}

\newcommand{\frg}{{\mathfrak g}}
\newcommand{\frh}{{\mathfrak h}}
\newcommand{\fra}{{\mathfrak a}}
\newcommand{\frb}{{\mathfrak b}}

\newcommand{\espan}[1]{\mathrm{span}\left\{#1\right\}}

\newcommand{\trace}{\mathrm{trace}}

\newcommand{\ZZ}{\mathbb{Z}}
\newcommand{\bZ}{{\mathbb Z}}

\newcommand{\QQ}{\mathbb{Q}}
\newcommand{\RR}{\mathbb{R}}

\newcommand{\bH}{{\mathbb H}}
\newcommand{\bK}{{\mathbb K}}
\newcommand{\OO}{\mathbb{O}}
\newcommand{\bO}{{\mathbb O}}
\newcommand{\bA}{{\mathbb A}}
\newcommand{\FF}{\mathbb{F}}
\newcommand{\bF}{{\mathbb F}}
\DeclareMathOperator{\Hom}{\mathrm{Hom}}
\DeclareMathOperator{\Mat}{\mathrm{Mat}}
\DeclareMathOperator{\Str}{\mathrm{Str}}
\DeclareMathOperator{\End}{\mathrm{End}}
\DeclareMathOperator{\Aut}{\mathrm{Aut}}%automorphism group
\DeclareMathOperator{\Der}{\mathrm{Der}}

\DeclareMathOperator{\supp}{\mathrm{Supp}\,}

\DeclareMathOperator{\Cent}{\mathrm{Cent}}
\DeclareMathOperator{\tor}{\mathrm{tor}}
\DeclareMathOperator{\rank}{\mathrm{rank}}

\newcommand{\ad}{\mathrm{ad}}

\newcommand{\frsl}{{\mathfrak{sl}}}

\newcommand{\frso}{{\mathfrak{so}}}

\newcommand{\frs}{{\mathfrak{s}}}

\newtheorem{theorem}{Theorem}[section]
\newtheorem{proposition}[theorem]{Proposition}

\theoremstyle{definition}
\newtheorem{df}[theorem]{Definition}

\theoremstyle{remark}
\newtheorem{remark}[theorem]{Remark}

\newenvironment{romanenumerate}
 {\begin{enumerate}
 
 }{\end{enumerate}}

\numberwithin{equation}{section}
\begin{document}

\title[Fine gradings and gradings by root systems]{Fine gradings and gradings by root systems\\
on simple Lie algebras}

%\author{A.E. to C.D.}
%\author[Cristina Draper]{Cristina Draper$^\star$}
%\thanks{$^{\star}$ Supported by ...}
%
%\address{Departamento de Matem\'atica Aplicada,
%Escuela de las Ingenier\'{\i}as, Universidad de M\'alaga,
%Ampliaci\'on del Campus de Teatinos,  29071 M\'alaga, Spain}
%\email{cdf@uma.es}

\author[Alberto Elduque]{Alberto Elduque$^{\star}$}
\thanks{$^{\star}$ Supported by the Spanish
Ministerio de Econom\'{\i}a y Competitividad and
FEDER (MTM2010-18370-C04-02) and by
the Diputaci\'on General de Arag\'on---Fondo Social
Europeo (Grupo de Investigaci\'on de \'Algebra)}

\address{Departamento de Matem\'aticas e
Instituto Universitario de Matem\'aticas y Aplicaciones,
Universidad de Zaragoza, 50009 Zaragoza, Spain}
\email{elduque@unizar.es}

%\date{\today}

%\subjclass[2010]{Primary , secondary .}

%\keywords{}

\begin{abstract}
Given a fine abelian group grading $\Gamma\,\colon\, \calL=\bigoplus_{g\in G}\calL_g$ of a finite dimensional simple Lie algebra over an algebraically closed field of characteristic zero, with $G$ being the universal grading group, it is shown that the induced grading by the free group $G/\tor(G)$ on $\calL$ is a grading by a (not necessarily reduced) root system.

Some consequences for the classification of fine gradings on the exceptional simple Lie algebras are drawn.
\end{abstract}

\maketitle

\setlength{\unitlength}{1mm}

%------------------------------------------------------------

\section{Introduction}

Gradings by abelian groups on simple Lie algebras appear in many instances. A systematic study of these gradings was started in \cite{PZ}. For the classical simple Lie algebra over an algebraically closed field of characteristic $0$, the fine gradings were classified in \cite{E10}. For the exceptional simple algebras they were classified in \cite{CristinaCandido_G2} and \cite{BahturinTvalavadze} for $G_2$, in \cite{CristinaCandido_F4} for $F_4$ and in \cite{CristinaViru} for $E_6$.

On the other hand, gradings by root systems were introduced by Berman and Moody in \cite{BermanMoody}, who used them as tools to study some classes of infinite-dimensional Lie algebras.

The goal of this paper is to relate both types of gradings. It will be shown that any fine grading with infinite universal grading group on a simple finite-dimensional Lie algebra over an algebraically closed field of characteristic $0$ induces a grading by a (possibly not reduced) root system. Some consequences for the classification of fine gradings in the exceptional cases will be derived too.

The first two sections will review the gradings by abelian groups and gradings by root systems respectively. The main result connecting fine gradings and gradings by root systems will be proved in the next two sections. This result shows that any fine grading is determined by a grading by a root system and a special grading on the coordinate algebra of the root grading. This grading on the coordinate algebra is studied in Section 5. The last section is devoted to draw consequences for the classification of the fine gradings on the simple exceptional simple Lie algebras.

%------------------------------------------------------------

\section{Gradings}

Let $\cA$ be an algebra (not necessarily associative) over a field $\FF$ and let $G$ be an abelian group (written additively).

\begin{df}\label{df:G_graded_alg}
A {\em $G$-grading} on $\cA$ is a vector space decomposition
\[
\Gamma\,\colon\,\cA=\bigoplus_{g\in G} \cA_g
\]
such that
\[
\cA_g \cA_h\subset \cA_{g+h}\quad\mbox{for all}\quad g,h\in G.
\]
If such a decomposition is fixed, we will refer to $\cA$ as a {\em $G$-graded algebra}.
The nonzero elements $a\in\cA_g$ are said to be {\em homogeneous of degree $g$}; we will write $\deg a=g$.
The {\em support} of $\Gamma$ is the set $\supp\Gamma\bydef\{g\in G\;|\;\cA_g\neq 0\}$.
\end{df}

Let
\[
\Gamma\,\colon\, \cA=\bigoplus_{g\in G} \cA_g\quad\text{and}\quad\Gamma'\,\colon\,\cB=\bigoplus_{h\in H} \cB_h
\]
be two gradings on algebras, with supports $S$ and $T$, respectively.

\begin{df}\label{df:equivalence}
We say that $\Gamma$ and $\Gamma'$ are {\em equivalent} if there exists an isomorphism of algebras $\psi\colon\cA\to\cB$ and a bijection $\alpha\colon S\to T$ such that $\psi(\cA_s)=\cB_{\alpha(s)}$ for all $s\in S$. Any such $\psi$ will be called an {\em equivalence} of $\Gamma$ and $\Gamma'$ (or of $\cA$ and $\cB$ if the gradings are clear from the context).
\end{df}

Given a group grading $\Gamma$ on an algebra $\cA$, there are many groups $G$ such that $\Gamma$, regarded as a decomposition into a direct sum of subspaces such that the product of any two of them lies in a third one, can be realized as a $G$-grading, but there is one distinguished group among them \cite{PZ}.

\begin{df}\label{df:universal_group}
Suppose that $\Gamma$ admits a realization as a $U$-grading for some abelian group $U$. We will say that $U$ is a {\em universal group of $\Gamma$} if, for any other realization of $\Gamma$ as a $G$-grading, there exists a unique homomorphism $U\to G$ that restricts to identity on $\supp\Gamma$.
\end{df}

One shows that the universal group, which we denote by $U(\Gamma)$, exists and depends, up to isomorphism, only on the equivalence class of $\Gamma$. Indeed, $U(\Gamma)$ is generated by $S=\supp\Gamma$ with defining relations $s_1+s_2=s_3$ whenever $0\ne\cA_{s_1}\cA_{s_2}\subset\cA_{s_3}$ ($s_i\in S$).

Given a $G$-grading $\Gamma\,\colon\,\cA=\bigoplus_{g\in G}\cA_g$ and a group homomorphism $\alpha\colon G\to H$, we obtain the induced $H$-grading ${}^\alpha\Gamma\,\colon\,\cA=\bigoplus_{h\in H}\cA'_h$ by setting $\cA'_h=\bigoplus_{g\in\alpha^{-1}(h)}\cA_g$.

\begin{df}\label{df:refinement_fine}
Given gradings $\Gamma\,\colon\,\cA=\bigoplus_{g\in G}\cA_g$ and $\Gamma'\,\colon\,\cA=\bigoplus_{h\in H}\cA'_h$, we say that $\Gamma'$ is a {\em coarsening} of $\Gamma$, or that $\Gamma$ is a {\em refinement} of $\Gamma'$, if for any $g\in G$ there exists $h\in H$ such that $\cA_g\subset\cA'_h$. The coarsening (or refinement) is said to be {\em proper} if the inclusion is proper for some $g\in \supp\Gamma$. (In particular, ${}^\alpha\Gamma$ is a coarsening of $\Gamma$, which is not necessarily proper.) A grading $\Gamma$ is said to be {\em fine} if it does not admit a proper refinement.
\end{df}

Any $G$-grading on a finite-dimensional algebra $\cA$ is induced from some fine grading $\Gamma$ by a homomorphism $\alpha\colon U(\Gamma)\to G$.

Over algebraically closed fields of characteristic zero, the classification of fine gradings on $\cA$ up to equivalence is the same as the classification of maximal diagonalizable subgroups (i.e., maximal quasitori) of $\Aut(\cA)$ up to conjugation (see e.g. \cite{PZ}). More precisely, given a grading $\Gamma$ on the algebra $\cA$ with universal group $G$, let $\hat G$ be its group of characters (homomorphisms $G\rightarrow \FF^\times$). Any $\chi\in \hat G$ acts as an automorphism of $\cA$ by means of $\chi.x=\chi(g)x$ for any $g\in G$ and $x\in \cA_g$. This allows us to identify $\hat G$ with a quasitorus (the direct product of a torus and a finite subgroup) of the algebraic group $\Aut(\cA)$. Conversely, given a quasitorus $Q$ of $\Aut(\cA)$, $Q=\hat G$ for $G$ the group of homomorphisms (as algebraic groups) $Q\rightarrow \FF^\times$. Then $Q$ induces a $G$-grading of $\cA$, where $\cA_g=\{x\in \cA: \chi(x)=g(\chi)x\}$ for any $g\in G$. In this way \cite{PZ}, the fine gradings on $\cA$, up to equivalence, correspond to the conjugacy classes in $\Aut(\cA)$ of the maximal quasitori (or maximal abelian diagonalizable subgroups) of $\Aut(\cA)$.

Fine gradings on simple Lie algebras belonging to the series $A$, $B$, $C$ and $D$ (including $D_4$) were classified in \cite{E10}. The fine gradings on the simple Lie algebra of type $G_2$ were classified in \cite{CristinaCandido_G2,BahturinTvalavadze}, for type $F_4$ in \cite{CristinaCandido_F4} (see also \cite{EK}), and for type $E_6$ in \cite{CristinaViru}.

\begin{df}\label{df:definitions_gradings}
Let $\Gamma\,\colon\,\cA=\bigoplus_{g\in G}\cA_g$ be a grading on the algebra $\cA$.
\begin{itemize}
\item A subspace $\cB$ of $\cA$ is said to be \emph{graded} if $\cB=\bigoplus_{g\in G}(\cB\cap\cA_g)$. (Equivalently, $\calB$ is graded by $G$ with $\cB_g=\cB\cap\cA_g$ for any $g\in G$.)
\item The \emph{type} of $\Gamma$ is the $r$-tuple $(n_1,\ldots,n_r)$, where $r=\max\{\dim\cA_g: g\in G\}$ and $n_i$ is the number of homogeneous components of dimension $i$, for $i=1,\ldots,r$. Hence $\dim\cA=\sum_{i=1}^rin_i$.
\end{itemize}
\end{df}

\smallskip
%\margen{Define special grading}

From now on, \emph{the ground field $\FF$ will be assumed to be algebraically closed of characteristic zero.}

%------------------------------------

\section{Gradings by root systems}

Berman and Moody \cite{BermanMoody} started the study of Lie algebras graded by root systems $\Phi$. (See \cite{ABG} and the references therein.)

\begin{df}\label{df:reduced_root_graded}
A Lie algebra $\calL$ over $\FF$ is \emph{graded by the reduced root system $\Phi$}, or  $\Phi$-graded, if:
\begin{romanenumerate}
\item $\calL$ contains as a subalgebra a finite-dimensional simple Lie algebra $\frg=\frh\oplus\bigl(\bigoplus_{\alpha\in\Phi}\frg_\alpha\bigr)$ whose root system is $\Phi$ relative to a Cartan subalgebra $\frh=\frg_0$;
\item $\calL=\bigoplus_{\alpha\in\Phi\cup\{0\}}\calL(\alpha)$, where $\calL(\alpha)=\{X\in\calL: [H,X]=\alpha(H)X\ \text{for all}$ $H\in\frh\}$; and
\item $\calL(0)=\sum_{\alpha\in\Phi}[\calL(\alpha),\calL(-\alpha)]$.
\end{romanenumerate}
The subalgebra $\frg$ is said to be a \emph{grading subalgebra} of $\calL$.
\end{df}

Berman and Moody \cite{BermanMoody} studied the simply laced case (types $A_r$, $D_r$ and $E_r$), and Benkart and Zelmanov \cite{BZ} considered the remaining  cases.

Under the adjoint action of $\frg$, a $\Phi$-graded Lie algebra $\calL$ decomposes as a sum of finite-dimensional irreducible $\frg$-modules whose highest weights are the highest long root, highest short root, or $0$. By collecting isomorphic summands into ``isotypic components'', we may assume that there are $\FF$-vector spaces $\calA$, $\calB$ and $\calD$ such that
\begin{equation}\label{eq:reduced_isotypic}
\calL=(\frg\otimes\calA)\oplus(\calW\otimes\calB)\oplus\calD,
\end{equation}
where the grading subalgebra $\frg$ is identified with $\frg\otimes 1$ for a distinguished element $1\in \calA$; $\calW$ is $0$ if $\frg$ is of type $A_r$ ($r\geq 1$), $D_r$ ($r\geq 4$), or $E_r$ ($r=6,7,8$), while $\calW$ is the irreducible $\frg$-module whose highest weight is the highest short root if $\frg$ is of type $B_r$ ($r\geq 2$), $C_r$ ($r\geq 3$), $F_4$ or $G_2$; and $\calD$ is the centralizer of $\frg\simeq \frg\otimes 1$, and hence it is a subalgebra of $\calL$.

The problem of classifying the $\Phi$-graded Lie algebras reduces to one of determining the possibilities for $\calA$, $\calB$ and $\calD$, and of finding the multiplication. The bracket in $\calL$ is invariant under the adjoint action of $\frg$ and this gives the sum $\fra=\calA\oplus\calB$ the structure of a unital algebra. Besides, $\calD$ acts as derivations on $\fra$, with $\calA$ and $\calB$ being invariant under this action. The type of the algebra $\fra$ depends on the root system $\Phi$. This algebra $\fra$ is called the \emph{coordinate algebra} of $\calL$.

For instance (see \cite{BZ}), assume that $\Phi$ is the root system of type $G_2$. Then $\frg$ is the Lie algebra of type $G_2$, which can be identified with the Lie algebra of derivations of the Cayley (or octonion algebra) $\OO$, and $\calW$ can be identified with the subspace of trace zero octonions $\OO_0$. The Cayley algebra is endowed with a nondegenerate quadratic form $n$ (the norm) such that any element $w$ satisfies $w^2-t(w)w+n(w)1=0$, where $t(w)=n(w,1)\bydef n(w+1)-n(w)-1$.

Moreover, one has the following properties:
\begin{enumerate}
\item $\Hom_\frg(\frg\otimes\frg,\frg)$ is spanned by the bracket,
\item $\Hom_\frg(\frg\otimes\frg,\FF)$ is spanned by the Killing form $\kappa$, which is a scalar multiple of the trace form relative to the representation provided by $\calW$.
\item $\Hom_\frg(\frg\otimes\calW,\calW)$ is spanned by the action of $\frg$ on $\calW$ ($X\otimes W\mapsto X.W$),
\item $\Hom_\frg(\calW\otimes \calW,\frg)$ is spanned by the map $w_1\otimes w_2\mapsto D_{w_1,w_2}$, where $D_{w_1,w_2}(w)= [[w_1,w_2],w]+3((w_1w)w_2-w_1(ww_2))$,
\item $\Hom_\frg(\calW\otimes\calW,\calW)$ is spanned by the bracket (inside $\OO$) $w_1\otimes w_2\mapsto [w_1,w_2]=w_1w_2-w_2w_1$.
\item $\Hom_\frg(\calW\otimes\calW,\FF)$ is spanned by the trace form $w_1\otimes w_2\mapsto t(w_1w_2)$.
\item $\Hom_\frg(\frg\otimes\frg,\calW)$, $\Hom_\frg(\frg\otimes\calW,\frg)$ and $\Hom_\frg(\frg\otimes\calW,\FF)$ are trivial.
\end{enumerate}

Therefore, the bracket in $\calL$ is given by:
\begin{equation}\label{eq:G2_graded}
\begin{split}
[D\otimes a,D'\otimes a']&=[D,D']\otimes a\cdot a'+\kappa(D,D')\langle a\vert a'\rangle,\\
[D\otimes a,w\otimes b]&=D(w)\otimes a\cdot b,\\
[d,D\otimes a]&=D\otimes d(a),\\
[d,w\otimes b]&=w\otimes d(b),\\
[w\otimes b,w'\otimes b']&=D_{w,w'}\otimes (b\vert b')+[w,w']\otimes b\circ b'+2t(w_1w_2)\langle b\vert b'\rangle,
\end{split}
\end{equation}
for any $D,D'\in\frg$, $w,w'\in\calW$, $a,a'\in \calA$, $b,b'\in \calB$ and $d,d'\in \calD$, and
for linear maps
\begin{itemize}
\item $\calA\otimes\calA\rightarrow \calA: a\otimes a'\mapsto a\cdot a'$, which is symmetric,
\item $\calA\otimes\calA\rightarrow \calD: a\otimes a'\mapsto \langle a\vert a'\rangle$, which is skew-symmetric,
\item $\calB\otimes\calB\rightarrow\calA: b\otimes b'\mapsto (b\vert b')$, which is symmetric,
\item $\calB\otimes\calB\rightarrow\calB: b\otimes b'\mapsto b\circ b'$, which is symmetric,
\item $\calB\otimes\calB\rightarrow \calD: b\otimes b'\mapsto \langle b\vert b'\rangle$, which is skew-symmetric,
\item $\calA\otimes\calB\rightarrow \calB: a\otimes b\mapsto a\cdot b$.
\end{itemize}
These linear maps satisfy the following properties:
\begin{enumerate}
\item $\calA$ is a unital commutative algebra with the product $a\cdot a'$,
\item $\fra=\calA\oplus\calB$ with the multiplication given by
\[
(a+b)\cdot(a'+b')=\bigl(a\cdot a'+(b\vert b')\bigr)+\bigl(a\cdot b'+a'\cdot b+b\circ b'\bigr),
\]
    for $a,a'\in\calA$ and $b,b'\in\calB$ is a Jordan algebra over $\calA$ with normalized trace given by $\trace(a+b)=a$, which satisfies the Cayley-Hamilton equation of degree $3$.
\item The action of $\calD$ on $\fra=\calA\oplus\calB$ is an action by derivations. Moreover, $\langle a\vert a'\rangle (\fra)=0=\langle b\vert b'\rangle(\calA)$ and $\langle b'\vert b''\rangle(b)=b'\cdot(b''\cdot b)-b''\cdot(b'\cdot b)$, for $a,a'\in\calA$ and $b,b',b''\in\calB$.
\item $\calD=\langle\calA\vert\calA\rangle +\langle\calB\vert\calB\rangle$. (This is imposed by condition (iii) in Definition \ref{df:reduced_root_graded})
\end{enumerate}

Therefore, in this case, the coordinate algebra $\fra$ is a Jordan algebra ``of degree $3$'' over the unital commutative associative algebra $\calA$ (see \cite{BZ}).

Note that $\langle \calA\vert\calA\rangle$ is a central ideal of $\calL$, so if $\calL$ is simple, then this is trivial, and hence $\calD=\langle \calB\vert\calB\rangle$.

\smallskip

Gradings by nonreduced root systems (type $BC_r$) will also appear attached to fine gradings. Following \cite{ABG} we recall the next definition:

\begin{df}\label{df:nonreduced_root_graded}
Let $\Phi$ be the nonreduced root system $BC_r$ ($r\geq 1$). A Lie algebra $\calL$ over $\FF$ is \emph{graded by $\Phi$}, or $\Phi$-graded, if:
\begin{romanenumerate}
\item $\calL$ contains as a subalgebra a finite-dimensional simple Lie algebra $\frg=\frh\oplus\bigl(\bigoplus_{\alpha\in\Phi'}\frg_\alpha\bigr)$ whose root system $\Phi'$ relative to a Cartan subalgebra $\frh=\frg_0$ is the reduced subsystem of type $B_r$, $C_r$ or $D_r$ contained in $\Phi$;
\item $\calL=\bigoplus_{\alpha\in\Phi\cup\{0\}}\calL(\alpha)$, where $\calL(\alpha)=\{X\in\calL: [H,X]=\alpha(H)X\,\ \text{for all}$ $H\in\frh\}$; and
\item $\calL(0)=\sum_{\alpha\in\Phi}[\calL(\alpha),\calL(-\alpha)]$.
\end{romanenumerate}
Again, the subalgebra $\frg$ is said to be a \emph{grading subalgebra} of $\calL$, and $\calL$ is said to be $BC_r$-graded with grading subalgebra of type $X_r$, where $X_r$ is the type of $\frg$.
\end{df}

Only $BC_r$-graded subalgebras of type $B_r$ will show up related to fine gradings on simple Lie algebras.

For $r\geq 3$, let $\calW$ be the natural module for the Lie algebra $\frg$ of type $B_r$. Thus $\calW$ is endowed with a symmetric nondegenerate bilinear form $(.\vert.)$, and
\[
\begin{split}
\frg&=\{x\in\End_\FF(\calW): (xu\vert v)=-(u\vert xv)\ \text{for all}\ u,v\in \calW\},\\
\frs&=\{s\in \End_\FF(\calW): (su\vert v)=(u\vert sv)\ \text{for all}\ u,v\in \calW\ \text{and}\ \trace(s)=0\}.
\end{split}
\]
In this case, a $BC_r$-graded subalgebra of type $B_r$ can be described, up to isomorphism, as follows (see \cite[(1.30)]{ABG}):
\begin{equation}\label{eq:BCr_isotypic}
\calL=(\frg\otimes\calA)\oplus(\frs\otimes\calB)\oplus(\calW\otimes\calC)\oplus\calD,
\end{equation}
The bracket in $\calL$ gives $\frb=\calA\oplus\calB\oplus\calC$ the structure of an algebra, which is termed the \emph{coordinate algebra} of $\calL$. Moreover (see \cite{ABG} for details), for $r\geq 3$ we have:
\begin{itemize}
\item The sum $\fra=\calA\oplus\calB$ is a unital associative algebra (multiplication denoted by $\alpha\cdot\alpha'$), with $1\in \calA$ (the subalgebra $\frg$ is identified with $\frg\otimes 1$), with involution $\eta$ whose subspace of symmetric elements is $\calA$ and whose subspace of skew-symmetric elements is $\calB$.
\item The space $\calC$ is an associative left $\fra$-module (action denoted by $\alpha\cdot c$, and it is equipped with a hermitian form $\xi$ relative to $\eta$, such that the multiplication in $\frb$ is given by:
    \[
    (\alpha +c)\cdot(\alpha'+c')=\bigl(\alpha\cdot\alpha'+\xi(c,c')\bigr)+\bigl(\alpha\cdot c'+\alpha'^\eta\cdot c\bigr).
    \]
\end{itemize}

For $r=2$, the grading subalgebra $\frb=\calA\oplus\calB\oplus\calC$ is a bit more involved, and can be described in terms of structurable algebras. (See \cite{ABG} for details.)

For $r=1$, a $BC_1$-graded subalgebra of type $B_1$ can be described, up to isomorphism, as follows:
\begin{equation}\label{eq:BC1_isotypic}
\calL=(\frg\otimes\calA)\oplus(\frs\otimes\calB)\oplus\calD.
\end{equation}
Here the natural module $\calW$ for the simple Lie algebra $\frg$ (isomorphic to $\frsl_2(\FF)$) of type $B_1$ is three-dimensional, and hence isomorphic to the adjoint module $\frg$, and the subspace of symmetric trace zero endomorphisms $\frs$ is the five-dimensional irreducible module for $\frg$.

In this case, results of Allison \cite{Allison_Models_Isotropic} give that the coordinate algebra $\fra=\calA\oplus\calB$ is a structurable algebra whose involution is given by $(a+b)^\eta=a-b$ (so $\calA$ is the subspace of symmetric elements and $\calB$ the subspace of skew-symmetric elements), and the quotient of $\calL$ by its center $Z(\calL)$ is the Tits-Kantor-Koecher Lie algebra constructed from the structurable algebra $(\fra,\eta)$. (See \cite[Theorem 2.6]{Benkart_Smirnov}.)

The arguments used in the proof of \cite[Theorem 7.5]{EO_Tits} give a more precise picture in this situation. The Lie bracket in $\calL$, which is invariant under the action of the subalgebra $\frg\simeq \frg\otimes 1$, is given by:

\begin{itemize}

\item $\calD$ is a subalgebra of $\calL$,

\smallskip

\item $[A\otimes a,B\otimes b]=[A,B]\otimes a\circ b\, -\,
\bigl(AB+BA-\frac{2}{3}\trace(AB)I_3\bigr)\otimes\frac{1}{2}[a,b]\,
+\frac{1}{2}\trace(AB)\langle a\vert b\rangle$,

\smallskip

\item $[A\otimes a,X\otimes x]=-(AX+XA)\otimes \frac{1}{2}[a,x]\,
+\, [A,X]\otimes a\circ x$,

\smallskip

\item $[X\otimes x,Y\otimes y]=[X,Y]\otimes x\circ y\
-\bigl(XY+YX-\frac{2}{3}\trace(XY)I_3\bigr)\otimes
\frac{1}{2}[x,y]\, +\frac{1}{2}\trace(XY)\langle x\vert y\rangle$,

\smallskip

\item $[d,A\otimes a]=A\otimes d(a)$,

\smallskip

\item  $[d,X\otimes x]=X\otimes d(x)$,
\end{itemize}

\noindent for any $A,B\in\frg$, $X,Y\in\frh$, $a,b\in\calA$,
$x,y\in\calB$, and $d\in \calD$, where

\begin{itemize}
\item $\calA\times\calA\rightarrow \calA$: $(a,b)\mapsto a\circ
b$ is a symmetric bilinear map with $1\circ a=a$ for any
$a\in\calA$,

\item $\calA\times\calA\rightarrow \calB$: $(a,b)\mapsto [a,b]$
is a skew symmetric bilinear map with $[1,a]=0$ for any $a\in\calA$,

\item $\calA\times\calB\rightarrow \calA$: $(a,x)\mapsto [a,x]$
is a bilinear map with $[1,x]=0$ for any $x\in\calB$,

\item $\calA\times\calB\rightarrow \calB$: $(a,x)\mapsto a\circ x$
is a bilinear map with $1\circ x=x$ for any $x\in\calB$,

\item $\calB\times \calB\rightarrow \calA$: $(x,y)\mapsto x\circ
y$ is a symmetric bilinear map,

\item $\calB\times\calB\rightarrow \calB$: $(x,y)\mapsto [x,y]$
is a skew symmetric bilinear map,

\item $\calA\times\calA\rightarrow \calD$: $(a,b)\mapsto \langle a\vert b\rangle$
is a skew symmetric bilinear map,

\item $\calB\times\calB\rightarrow \calD$: $(x,y)\mapsto \langle x\vert y\rangle$
is a skew symmetric bilinear map,

\item the bilinear maps $\calD\times\calA\rightarrow\calA$:
$(d,a)\mapsto d(a)$ and $\calD\times\calB\rightarrow \calB$:
$(d,x)\mapsto d(x)$, give two representations of the Lie algebra
$\calD$.

\end{itemize}

Define $x\circ a=a\circ x$ and $[x,a]=-[a,x]$ for any $a\in \calA$
and $x\in\calB$, and define on the vector space
$\fra=\calA\oplus\calB$ a multiplication  by means
of
\[
u\cdot v=u\circ v+\frac{1}{2}[u,v]
\]
for any $u,v\in \calA\cup\calB$, so $u\circ v=\frac{1}{2}(u\cdot
v+v\cdot u)$ and $[u,v]=u\cdot v-v\cdot u$. Define too a linear map
$-\colon \fra\rightarrow \fra$ such that $\overline{a+x}=a-x$ for any
$a\in\calA$ and $x\in\calB$. Then (\cite[Theorem 7.5]{EO_Tits}) the subspace $\fra$, with this multiplication and involution, is a structurable algebra.

Besides, condition (iii) in Definition \ref{df:nonreduced_root_graded} shows $\calD=\langle\calA\vert\calA\rangle +\langle\calB\vert\calB\rangle$, and a straightforward application of the Jacobi identity gives
\[
\langle a\vert b\rangle (u)=D_{a,b}(u),\qquad \langle x\vert y\rangle (u)=D_{x,y}(u),
\]
for any $a,b\in\calA$, $x,y\in\calB$ and $u\in \calA\cup\calB$, where $D_{u,v}$ is the derivation of the structurable algebra $\fra$ defined in \cite[Equation (15)]{Allison78}:
\begin{equation}\label{eq:Duv}
D_{u,v}(w)=\frac{1}{3}[[u,v]+[\bar u,\bar v],w]+(w,v,u)-(w,\bar u,\bar v),
\end{equation}
for $u,v,w\in\fra$, where $(w,v,u)=(w\cdot v)\cdot u-w\cdot (v\cdot u)$ is the associator of the elements $w,v,u$.

%-------------------------------------

\section{Fine gradings on semisimple Lie algebras}

The aim of this section is to show that any fine grading on a finite dimensional semisimple Lie algebra, with the property that the free rank of its universal group is $>0$, determines in a natural way a (possibly non reduced) root system. This root system is irreducible if the Lie algebra is simple.

The first two items of the next Proposition have been proved in \cite{CristinaCandido_F4} over the field of complex numbers. Given a finitely generated abelian group $G$, let $\tor(G)$ denote its torsion subgroup and let $\bar G$ be the quotient $G/\tor(G)$, which is free. Its rank is called the free rank of $G$.

\begin{proposition}\label{pr:fine_weight}
Let $\calL$ be a finite dimensional semisimple Lie algebra and let $\Gamma\,\colon\,\calL=\bigoplus_{g\in G} \calL_g$ be a fine grading. Assume that $G$ is the universal group of $\Gamma$. (Since the dimension of $\calL$ is finite, $G$ is a finitely generated abelian group.)

Then the following conditions hold:
\begin{romanenumerate}
\item The neutral homogeneous component $\calL_0$ is a toral subalgebra of $\calL$ (i.e., $\ad \calL_0$ consists of commuting diagonalizable operators in $\calL$).
\item The dimension of $\calL_0$ coincides with the free rank of $G$.
\item Let $\tor(G)$ be the torsion subgroup of $G$. The induced grading $\bar\Gamma\,\colon\, \calL=\bigoplus_{\bar g\in G/\tor(G)} \calL_{\bar g}$ is the weight space decomposition relative to $\calL_0$.
\end{romanenumerate}
\end{proposition}

\begin{proof}
The Killing form of $\calL$ satisfies $\kappa(\calL_g,\calL_h)=0$ unless $g+h=0$, and hence the restriction of $\kappa$ to $\calL_0$ is nondegenerate. This shows that $\calL_0$ is reductive (see \cite[Chapter I, \S6.4, Proposition 5]{Bou98}). Moreover, for any $X\in Z(\calL_0)$ (the center of $\calL_0$), the semisimple and nilpotent parts of $X$ belong to $Z(\calL_0)$ too and $\kappa(X_n,\calL_0)=0$ since $\ad X_n$ is nilpotent, so we get $X_n=0$. Therefore, the elements of $Z(\calL_0)$ are semisimple and $\calL_0$ is reductive in $\calL$ (see \cite[Chapter 1, \S 6.4,6.5]{Bou98}).

Let $\frh$ be a Cartan subalgebra of $\calL_0$. Hence $Z(\calL_0)$ is contained in $\frh$ and $\frh$ is maximal among the toral subalgebras of $\calL$ contained in $\calL_0$. For any $g\in G$, $\calL_g$ is invariant under the adjoint action of $\calL_0$. Therefore, $\Gamma$ can be refined by means of the weight space decomposition relative to the toral subalgebra $\frh$.

Since $\Gamma$ is fine, for any $g\in G$ there exists a linear form $\alpha\in\frh^*$ such that $\calL_g$ is contained in the weight space $\calL(\alpha)\bydef \{X\in \calL: [H,X]=\alpha(H)X\ \text{for all $H\in \frh$}\}$. In particular, $\calL_0=\calL(0)\cap\calL_0=\frh$ is a toral subalgebra. This proves the first part. (Note that $0$ denotes both the neutral component of $G$ and the trivial linear form, but this should cause no confusion.)

Therefore, $\Gamma$ is a refinement of the grading given by the weight space decomposition relative to the toral subalgebra $\frh=\calL_0$: $\hat\Gamma\,\colon\,\calL=\bigoplus_{\alpha\in\frh^*}\calL(\alpha)$. Denote by $\Phi$ the set of nonzero weights in this decomposition:
\begin{equation}\label{eq:Phi}
\Phi\bydef \{\alpha\in \frh^*\setminus{0}: \calL(\alpha)\ne 0\}.
\end{equation}
Then $\ZZ\Phi$ is a free abelian subgroup of $\frh^*$ and we may look at $\hat\Gamma$ as a grading by the group $\ZZ\Phi$.

Since $G$ is the universal group of $\Gamma$ and $\hat\Gamma$ is a coarsening of $\Gamma$, there is a surjective homomorphism \begin{equation}\label{eq:pi}
\pi \colon G\rightarrow \ZZ\Phi
\end{equation}
such that $\pi(g)=\alpha$ if $\calL_g\subseteq \calL(\alpha)$. And since $\ZZ\Phi$ is torsion free, $\pi$ induces a surjective homomorphism $\bar\pi\colon \bar G\bydef G/\tor(G)\rightarrow \ZZ\Phi$. In particular, the rank of the free group $\bar G$ is greater than or equal to the rank of $\ZZ\Phi$.

But $\FF \Phi$ is the whole dual vector space $\frh^*$, as otherwise there would exist an element $0\ne X\in \frh$ such that $\alpha(X)=0$ for any $\alpha\in \Phi$, and then $X$ would belong to the center of $\calL$, and this is trivial since $\calL$ is semisimple. In particular, this shows that the rank of the free abelian group $\ZZ\Phi$ is greater than or equal to the dimension of the vector space $\FF\Phi=\frh^*$. Hence we obtain $\rank(\ZZ \Phi)\geq \dim\frh$, and thus $\rank\bar G\geq \dim\frh$.

Since the universal group $G$ is generated by the support of $\Gamma$, so is $\bar G$ generated by the support of $\bar\Gamma$. But $\bar G$ is a finitely generated free abelian group, so there are elements $\bar g_1,\ldots,\bar g_m\in \supp\bar\Gamma$ such that $\bar G=\ZZ\bar g_1\oplus\cdots\oplus\ZZ\bar g_m$ (here $\bar g$ denotes the class of $g$ modulo $\tor(G)$).

The Lie algebra $\calL$ is semisimple, and hence any derivation is inner. In particular, for any $i=1,\ldots,m$, there is a unique element $H_i\in \calL$ such that $[H_i,X]=n_iX$ for any $X\in \calL_{n_1\bar g_1+\cdots+n_m\bar g_m}$. Moreover, we may replace $H_i$ by its component in $\calL_0=\frh$ for any $i$ so, by uniqueness, we obtain $H_1,\ldots,H_m\in \calL_0$. Since the sum $\calL_{\bar g_1}\oplus\cdots\oplus\calL_{\bar g_m}$ is direct, the elements $H_1,\ldots,H_m$ are linearly independent, and hence we get $m=\rank\bar G\leq \dim\frh$. This proves the second part: $\rank\bar G=\dim\frh$.

\smallskip

The argument above shows that $\frh=\FF H_1\oplus\cdots\oplus\FF H_m$, and for any $\bar g=n_1\bar g_1+\cdots+n_m\bar g_m$ we have $\calL_{\bar g}=\calL(\alpha)$, where $\alpha$ is the linear form on $\frh$ such that $\alpha(H_i)=n_i$ for any $i$. This proves the last part.
\end{proof}

\begin{remark}\label{rk:frg(0)}
The neutral component of the grading $\bar \Gamma$ in Proposition \ref{pr:fine_weight} is $\calL_{\bar 0}=\bigoplus_{g\in \tor(G)}\calL_g=\calL(0)$, and this is the centralizer $\Cent_{\calL}(\calL_0)$.

Using the arguments in the proof above, $\calL(0)$ is shown to be reductive in $\calL$, so $\calL(0)=Z(\calL(0))\oplus [\calL(0),\calL(0)]$ and $\calL_0\subseteq Z(\calL(0))$. In particular, the neutral component of the restriction of $\Gamma$ to $[\calL(0),\calL(0)]$ is trivial: $[\calL(0),\calL(0)]_0=0$.

Therefore, $\Gamma$ induces a grading on $[\calL(0),\calL(0)]$ by the finite group $\tor(G)$ whose homogeneous component of degree $0$ is trivial. These gradings are called \emph{special}. (See \cite{Hesselink} for properties of these gradings.)
\end{remark}

\begin{remark}\label{rk:(iii)not_enough}
Condition (ii) in Proposition \ref{pr:fine_weight} does not suffice to ensure that the grading $\Gamma$ is fine. As an example, let $V$ be a five-dimensional vector space with a basis $\{e_1,e_2,e_3,e_4,e_5\}$, endowed with the symmetric bilinear form $b$ such that $b(e_i,e_j)=\delta_{ij}$ for any $i,j$. Consider the associated orthogonal Lie algebra $\frso(V,b)$ of skew symmetric endomorphisms relative to $b$. The vector space $V$ is graded by $\bZ_2^3$, with $\deg e_1=(\bar 1,\bar 0,\bar 0)$, $\deg e_2=(\bar 0,\bar 1,\bar 0)$, $\deg e_3=(\bar 0,\bar 0,\bar 1)$, $\deg e_4=(\bar 1,\bar 1,\bar 1)$ and $\deg e_5=0$. This induces a grading by $\bZ_2^3$ on $\frso(V,b)$ of type $(4,3)$, with the basic elements $E_{ij}-E_{ji}$, $i\ne j$, being homogeneous of degree $\deg e_i+\deg e_j$. Here $E_{ij}$ denotes the endomorphism that takes $e_j$ to $e_i$ and annihilates the other basic elements. Then $\frsl(V,b)_0=0$, the free rank of the finite grading group is also $0$, but this grading is not fine, as it can be refined to get a grading of type $(10)$ by $\bZ_2^4$ with $\deg e_1=(\bar 1,\bar 0,\bar 0,\bar 0)$, $\deg e_2=(\bar 0,\bar 1,\bar 0,\bar 0)$, $\deg e_3=(\bar 0,\bar 0,\bar 1,\bar 0)$, $\deg e_4=(\bar 0,\bar 0,\bar 0,\bar 1)$ and $\deg e_5=0$.
\end{remark}

\begin{theorem}\label{th:root_system}
Let $\calL$ be a finite dimensional semisimple Lie algebra and let $\Gamma\,\colon\,\calL=\bigoplus_{g\in G} \calL_g$ be a fine grading. Assume that $G$ is the universal group of $\Gamma$. Let $\Phi$ be as in \eqref{eq:Phi}. Then,
$\Phi$ is a (possibly non reduced) root system in the euclidean vector space $E=\RR\otimes_\QQ \QQ\Phi$. If $\calL$ is simple, then $\Phi$ is an irreducible root system.
\end{theorem}
\begin{proof}
Several steps will be followed:

\noindent\textbf{1.}\quad
Because of Proposition \ref{pr:fine_weight}, the set of weights $\Phi$ is precisely $\pi(\supp\Gamma\setminus\tor(G))$, with $\pi$ in \eqref{eq:pi}. Hence, for any $g\in \supp\Gamma\setminus\tor(G)$, let $\alpha=\pi(g)$ and take an element $0\ne X\in \calL_g\subseteq \calL(\alpha)$. Then $\calL_{-g}$ is contained in $\calL(-\alpha)$. Since $\alpha$ is not $0$, $\ad X$ is nilpotent. By the Jacobson--Morozov Theorem \cite[Chapter III, Theorem 17]{Jacobson}, there are elements $H,Y\in \calL$ such that $[H,X]=2X$, $[H,Y]=-2Y$ and $[X,Y]=H$ (i.e.; $X,H,Y$ form an \emph{$\frsl_2$-triple}). We have $H=\sum_{h\in G}H_h$, $Y=\sum_{h\in G}Y_h$ for homogeneous elements $H_h,Y_h\in \calL_h$, $h\in G$. Then $[H,X]=2X$ implies $[H_0,X]=2X$, so $\alpha(H_0)=2$, and hence $[H_0,Y_{-g}]=-2Y_{-g}$. Also, from $[X,Y]=H$ we get $[X,Y_{-g}]=H_0$. Therefore, we may take $H\in \calL_0=\frh$ and $Y\in \calL_{-g}$.

\noindent\textbf{2.}\quad
The restriction of the Killing form $\kappa$ to $\frh=\calL_0$ is nondegenerate, so it induces a nondegenerate symmetric bilinear form $(.\mid .)$ on $\frh^*=\FF \Phi$.  For any $\alpha\in \Phi$, take an element $g\in G$ with $\pi(g)=\alpha$, and an $\frsl_2$-triple $X\in \calL_g$, $H\in \calL_0$, $Y\in\calL_{-g}$ as above. For any $\beta\in \Phi$, the sum $\bigoplus_{i\in\ZZ}\calL(\beta+i\alpha)$ is a module for the subalgebra $\frs=\espan{X,H,Y}$ (isomorphic to $\frsl_2(\FF)$). With standard arguments we obtain $\beta(H)=r-q\in\ZZ$ and $\beta-\beta(H)\alpha\in \Phi$, where $q=\max\{n\in\ZZ : \beta+n\alpha\in \Phi\}$, $r=\max\{n\in \ZZ : \beta -n\alpha\in \Phi\}$. In particular, $H_\alpha\bydef H$ does not depend on $g$ or $X$, only on $\alpha$. Also, we get
\[
\kappa(H_\alpha,H_\alpha)=\sum_{\beta\in\Phi}\bigl(\dim\calL(\beta)\bigr)\beta(H_\alpha)^2\in \ZZ_{>0}.
\]

\noindent\textbf{3.}\quad For any $\alpha\in \frh^*$ there is a unique $T_\alpha\in \frh$ such that $\alpha(H)=\kappa(T_\alpha,H)$ for any $H\in\frh$.
If the element $T\in \frh=\calL_0$ satisfies $\alpha(T)(=\kappa(T_\alpha,T))=0$, then for any $\beta\in \Phi$ we have
\[
\trace\left((\ad H_\alpha \ad T)\vert_{\bigoplus_{i\in \ZZ}\calL_{\beta+i\alpha}}\right)
= \beta(T)\trace\left(\ad H_\alpha\vert_{\bigoplus_{i\in \ZZ}\calL_{\beta+i\alpha}}\right)=0.
\]
Hence $\kappa(H_\alpha,T)=0$ too, and hence $H_\alpha\in \FF T_\alpha$. Since $\alpha(H_\alpha)=2$, we get $H_\alpha=\frac{2}{\kappa(T_\alpha,T_\alpha)}T_\alpha =\frac{2}{(\alpha\vert\alpha)}T_\alpha$. Define, as usual, $\langle\beta\vert\alpha\rangle\bydef \frac{2(\beta\vert\alpha)}{(\alpha\vert\alpha)}=\beta(H_\alpha)$. Therefore we have for any $\alpha,\beta\in \Phi$ that
\[
\langle\beta\vert\alpha\rangle\in \ZZ\quad\text{and}\quad\beta-\langle \beta\vert\alpha\rangle\alpha\in \Phi.
\]
Also we have $\kappa(H_\alpha,H_\alpha) =\frac{4}{(\alpha\vert\alpha)^2}\kappa(T_{\alpha},T_{\alpha})
=\frac{4}{(\alpha\vert\alpha)}$, so $(\alpha\vert\alpha)=\frac{4}{\kappa(H_\alpha,H_\alpha)}$ is a positive rational number.

\noindent\textbf{4.}\quad
Take a basis $\{\alpha_1,\ldots,\alpha_m\}$ of $\frh^*$ contained in $\Phi$, and let $g_1,\ldots,g_m$ be elements in $G$ with $\pi(g_i)=\alpha_i$ for any $i=1,\ldots,m$. For any $\gamma\in \QQ\Phi\,(\subseteq \frh^*)$, there are rational numbers $r_1,\ldots,r_m$ such that $\gamma=r_1\alpha_1+\cdots+r_m\alpha_m$, and we get:
\[
\begin{split}
(\gamma\vert\gamma)&=\kappa(T_\gamma,T_\gamma)=\sum_{\beta\in \Phi}\bigl(\dim\calL(\beta)\bigr)\beta(T_\gamma)^2\\
 &=\sum_{\beta\in \Phi}\bigl(\dim\calL(\beta)\bigr)\left(\sum_{i=1}^m r_i\beta(T_{\alpha_i}\right)^2\\
 &=\sum_{\beta\in \Phi}\bigl(\dim\calL(\beta)\bigr)\left(\sum_{i=1}^m r_i(\beta\vert\alpha_i)\right)^2\\
 &=\sum_{\beta\in \Phi}\bigl(\dim\calL(\beta)\bigr)\left(\sum_{i=1}^m \frac{r_i(\alpha_i\vert\alpha_i)}{2}\langle\beta\vert\alpha_i\rangle\right)^2\in\QQ_{>0}.
\end{split}
\]
Hence $E=\RR\otimes_\QQ\QQ\Phi$ is a euclidean vector space with inner product determined by $(.\mid .)$, $\Phi$ is a finite subset of $E$ not containing $0$, that spans $E$ and such that $\langle \alpha\vert\beta\rangle=\frac{2(\alpha\vert\beta)}{(\beta\vert\beta)}\in \ZZ$ and $\beta-\langle\beta\vert\alpha\rangle\alpha\in \Phi$, for any $\alpha,\beta\in \Phi$. Therefore, $\Phi$ is a root system.

\noindent\textbf{5.}\quad
If $\calL$ is simple, then $\Phi$ must be irreducible, as otherwise $\Phi$ would split as a disjoint union $\Phi=\Phi_1\dot\cup \Phi_2$, with $(\Phi_1\vert\Phi_2)=0$. But then $\bigl(\bigoplus_{\alpha\in\Phi_1}\calL(\alpha)\bigr)\oplus\bigl(\sum_{\alpha\in \Phi_1}[\calL(\alpha),\calL(-\alpha)]\bigr)$ would be a proper ideal of $\calL$.
\end{proof}

%------------------------------

\section{The main result}

With the same hypotheses as in the previous section, take a system of simple roots $\Delta$ of the root system $\Phi$ in \eqref{eq:Phi}. Hence $\Delta$ is a basis of $\frh^*$ contained in $\Phi$ and $\Phi=\Phi^+\dot\cup \Phi^-$, with $\Phi^+\subseteq \sum_{\alpha\in\Delta}\ZZ_{\geq 0}\alpha$, $\Phi^-=-\Phi^+$. For any $\alpha\in \Delta$ choose $g_\alpha\in G$ such that $\pi(g_\alpha)=\alpha$. Since $G$ is generated by $\supp\Gamma$, we have $G=\bigl(\bigoplus_{\alpha\in \Delta}\ZZ g_\alpha\bigr)\oplus \tor(G)$. Let $\tilde G\bydef \bigoplus_{\alpha\in \Delta}\ZZ g_\alpha$ and let
\begin{equation}\label{eq:frg}
\frg\bydef \bigoplus_{g\in \tilde G}\calL_g.
\end{equation}

The arguments in the proof of Proposition \ref{pr:fine_weight} show that $\frg$ is a reductive subalgebra in $\calL$. Also, any $0\ne X\in \frg_g$, $g\ne 0$, is contained in a $\frsl_2$-triple, so the center $Z(\frg)$ is contained in $\calL_0=\frh$. But the dimension of $\frh$ coincides with the rank of $\ZZ \Phi$, so we conclude that $Z(\frg)=0$ and $\frg$ is semisimple.

Also, any weight of $\frh$ on $\frg$ belongs to $\pm\bigl(\bigoplus_{\alpha\in \Delta}\ZZ_{\geq 0}\alpha\bigr)$, so $\Delta$ is a system of simple roots for $\frg$ relative to its Cartan subalgebra $\frh$. We conclude that $\frg$ is, up to isomorphism, the semisimple Lie algebra with $\Delta$ as a system of simple roots.

Now the main result of the paper, relating fine gradings and gradings by root systems, follows easily:

\begin{theorem}\label{th:root_grading}
Let $\calL$ be a finite dimensional simple Lie algebra and let $\Gamma\,\colon\,\calL=\bigoplus_{g\in G} \calL_g$ be a fine grading. Assume that $G$ is the universal group of $\Gamma$. Let $\Phi$ be as in \eqref{eq:Phi}.
Then $\calL$ is graded by the irreducible (possibly nonreduced) root system $\Phi$ with grading subalgebra $\frg$ in \eqref{eq:frg}. Moreover, if $\Phi$ is nonreduced (type $BC_r$), then $\frg$ is simple of type $B_r$.
\end{theorem}
\begin{proof}
The Lie algebra $\calL$ contains the semisimple subalgebra $\frg$ with Cartan subalgebra $\frh$ and system of simple roots $\Delta$. Since $\calL$ is simple, $\Phi$ (or $\Delta$) is irreducible, and the ideal $\bigl(\bigoplus_{\alpha\in\Phi}\calL(\alpha)\bigr)\oplus\bigl(\sum_{\alpha\in \Phi}[\calL(\alpha),\calL(-\alpha)]\bigr)$ is the whole $\calL$. Hence $\calL$ is graded by the root system $\Phi$ with $\frg$ as a grading subalgebra. Moreover, any root in $\Phi$ is a sum of roots in $\frg$. Hence for $\Phi$ of type $BC_n$, $\frg$ is of type $B_n$.
\end{proof}

%-----------------------------------------

\section{Grading on the coordinate algebra}

Let $\Gamma\,\colon\, \calL=\bigoplus_{g\in G}\calL_g$ be a fine grading on a finite-dimensional simple Lie algebra, with $G$ being the universal group of $\Gamma$. As in the proof of Proposition \ref{pr:fine_weight}, let $\Phi$ be the set of weights of the adjoint action of $\calL_0$, and let $\pi\colon G\rightarrow \ZZ\Phi$ be the surjective group homomorphism with $\pi(g)=\alpha$ if $\calL_g\subseteq\calL(\alpha)$. Then $\pi$ induces an isomorphism $\bar\pi\colon \bar G=G/\tor(G)\rightarrow \ZZ\Phi$ by item (iii) of Proposition \ref{pr:fine_weight}. Let $\frg$ be the grading subalgebra in Theorem \ref{th:root_grading}, obtained after fixing a system of simple roots $\Delta$ and preimages $g_\alpha$ under $\pi$ of the elements in $\Delta$. Also, consider the free abelian group $\tilde G=\bigoplus_{\alpha\in\Delta}\ZZ g_\alpha$, such that $G=\tilde G\oplus\tor(G)$. The restriction of $\pi$ to $\tilde G$ is bijective.

If $\Phi$ is reduced, then we have a decomposition as in equation \eqref{eq:reduced_isotypic}. Then:
\begin{itemize}
\item $\frg=\frg\otimes 1$ is, by its own construction, a graded subalgebra of $\calL$, and hence so is its centralizer $\calD=\Cent_\calL(\frg)$.\\
    Besides, $\calD$ is contained in $\calL(0)=\bigoplus_{g\in \tor(G)}\calL_g$, and hence $\calD$ is graded by $\tor(G)$. Moreover, $\calD_0\subseteq \calL_0=\frg_0$, so $\calD_0=0$.\\
    Therefore, the grading of $\calD$ by $\tor(G)$ is a special grading. (See Remark \ref{rk:frg(0)}.)

\item Let $\lambda$ be the highest root of $\frg$ (relative to $\Delta$), then $\lambda$ is not a weight of $\calW$, and hence $\calL(\lambda)=\frg_\lambda\otimes\calA$. On the other hand, if $g_\lambda$ is the preimage in $\tilde G$ of $\lambda$, then
    \[
    \calL(\lambda)=\bigoplus\{\calL_g: \pi(g)=\lambda\}=\bigoplus_{g\in \tor(G)}\calL_{g_\lambda +g},
    \]
    so the vector space $\calA$ is graded by $\tor(G)$, where $\calA_h$ is defined by means of:
    \begin{equation}\label{eq:Ah}
    \calL_{g_\lambda+h}=\frg_\lambda\otimes\calA_h,
    \end{equation}
    for any $h\in\tor(G)$.

\item Since $\frg\otimes\calA$ is the $\frg$-submodule of $\calL$ generated by $\frg_\lambda\otimes\calA$ ($\lambda$ is the highest root of $\frg$), it follows that $\frg\otimes\calA$ is a graded subspace of $\calL$ and for any $g\in \tilde G$ and $h\in \tor(G)$ we have
    \[
    (\frg\otimes\calA)_{g+h}=\frg_{\pi(g)}\otimes\calA_h.
    \]

\item By invariance under the adjoint action of $\frg$, the subspace $\calW\otimes\calB$ turns out to be the orthogonal complement of $\bigl(\frg\otimes\calA\bigr)\oplus\calD$ relative to the Killing form of $\calL$. Since this latter subspace is a graded subspace of $\calL$, so is $\calW\otimes\calB$. \\
    Let $\mu$ be the highest weight of the $\frg$-module $\calW$ relative to $\Delta$ ($\mu$ is the highest short root). Let $g_\mu$ be the preimage by $\pi$ in $\tilde G$ of $\mu$. Then, as for $\calA$, we also get that $\calB$ is graded by $\tor(G)$ if we define $\calB_h$ by means of
    \[
    (\calW\otimes\calB)_{g_\mu+h}=\calW_\mu\otimes \calB_h,
    \]
    for any $h\in \tor(G)$. And since $\calW$ is generated, as a module for $\frg$, by $\calW_\mu\,(=\{w\in \calW: H.w=\mu(H)w\ \text{for all}\ H\in\frh=\frg_0\})$, it follows that the subspace $\calW\otimes\calB$ is a graded subspace of $\calL$ and for any $g\in\tilde G$ and $h\in \tor(G)$ we have
    \[
    (\calW\otimes\calB)_{g+h}=\calW_{\pi(g)}\otimes\calB_h.
    \]
\end{itemize}

On the other hand, if $\Phi$ is nonreduced of type $BC_1$, then we have a decomposition as in Equation \eqref{eq:BC1_isotypic}, and the same arguments show that $\calD$ inherits a special grading by $\tor(G)$, that if $\mu$ is the highest weight then $\calL(\mu)=\calW_\mu\otimes\calB$, and this shows that $\calB$ is graded by $\tor(G)$ as above. Finally, $\frg\otimes\calA$ is the orthogonal complement to $(\calW\otimes\calB)\oplus\calD$ relative to the Killing form, and we conclude that $\calA$ is graded too by $\tor(G)$ as above.

Finally, if $\Phi$ is nonreduced of type $BC_r$, $r\geq 2$, then we have a decomposition as in Equation \eqref{eq:BCr_isotypic} and one checks as before that $\calD$ inherits a special grading by $\tor(G)$, that if $\mu$ is the highest weight of $\frs$, then $\calL(\mu)=\frs_\mu\otimes\calB$, and hence it follows that $\calB$ is $\tor(G)$-graded. Then $(\frg\otimes\calA)\oplus(\calW\otimes\calC)$ is the orthogonal complement, so it is a graded subspace too. Here, if $\lambda$ is the highest root, then $\Bigl((\frg\otimes\calA)\oplus(\calW\otimes\calC)\Bigr)\cap\calL(\lambda)=\frg_\lambda\otimes \calA$, so again we conclude that $\calA$ is $\tor(G)$-graded, and from here we deduce that so is $\calC$.

These arguments prove most of the next result:

\begin{proposition}\label{pr:grading_coordinate_algebra}
Under the conditions above, with $\Phi$ being an irreducible root system, the coordinate algebra $\fra=\calA\oplus\calB$ (in the reduced case or for $BC_1$) or $\frb=\fra\oplus\calC$ (in the $BC_r$-case, $r\geq 2$) inherits a fine grading by $\tor(G)$, where $\cA$ and $\cB$, and $\cC$ in the $BC_r$-case, $r\geq 2$, are graded subspaces.

Moreover, $\calA_0=\FF 1$ while $\calB_0=0$, and also $\calC_0=0$ (in the $BC_r$-case, $r\geq 2$), $\tor(G)$ is the universal group, and this grading on $\fra$, or $\frb$, induces a special grading on $\calD$ by $\tor(G)$.
\end{proposition}
\begin{proof}
The fact that $\fra$ inherits a grading by $\tor(G)$ is clear from the earlier arguments. Also $\calL_0=\frg_0=\frg_0\otimes 1$, so $\calA_0=\FF 1$ and $\calB_0=0$ (and $\calC_0=0$ too in the $BC_r$ case, $r\geq 2$). Hence $\fra_0=\FF 1$. Besides, any refinement of this grading on $\fra$ would give a refinement of $\Gamma$, as the grading by $\tor(G)$ of $\calD$ is determined by the grading on $\fra$, because of condition (iii) in Definition \ref{df:reduced_root_graded}. The last part is a direct consequence of $G$ being the universal group of $\Gamma$.
\end{proof}

%----------------------------------

\section{Applications}

The results in the previous sections will be used to classify the fine gradings on the simple exceptional Lie algebras whose universal group have free rank $>2$. Quick proofs of the classification of fine gradings, up to equivalence, on the simple Lie algebras of type $G_2$ and $F_4$ will be given too.

\smallskip

Table 25 in \cite{Dynkin} gives a list of the simple subalgebras of rank $>1$ of the exceptional simple Lie algebras, together with the decomposition of any such simple Lie algebra as a sum of irreducible modules for the simple subalgebra. This immediately gives the different possibilities, up to conjugation, of grading an exceptional simple Lie algebra by an irreducible (not necessarily reduced) root system of rank $\geq 2$. The different possibilities are summarized in Table \ref{table:root_gradings}, where $\frg$, $\frs$, $\calW$, $\calA$, $\calB$, $\calC$ and $\calD$ are as in Equations \eqref{eq:reduced_isotypic} or \eqref{eq:BCr_isotypic}.

In many cases, this corresponds (see \cite{BZ,Tits66}) to the well-known Tits construction $\calT(\calC,\calJ)$, for a unital composition algebra $\calC$ and a degree three simple Jordan algebra $\calJ$, which we recall now (see also \cite{EO_Tits}):

Let $\calH$ be a unital composition algebra (or Hurwitz algebra)
with norm $n$ and trace $t$. The unital composition algebras are, up to isomorphism, $\bF$, $\bK=\bF\oplus\bF$, $\bH=\Mat_2(\bF)$ (quaternion algebra), and the Cayley algebra $\bO$ (recall that the ground field $\bF$ is assumed to be algebraically closed).
Let $\calJ$ be a unital simple Jordan algebra of degree $3$, so that $\calJ$ is the Jordan algebra $H_3(\calH')$ of hermitian $3\times 3$ matrices over another unital composition algebra $\calH'$. Denote by $\calH_0$ and $\calJ_0$ the subspaces of trace zero elements in $\calH$ and $\calJ$.

For $a,b\in\calH$, the linear map $D_{a,b}\colon \calH\rightarrow\calH$ defined by $D_{a,b}(c)=[[a,b],c]+3(a,c,b)$,
where $[a,b]=ab-ba$ is the commutator, and $(a,c,b)=(ac)b-a(cb)$ the
associator, is a derivation of $\calH$. These derivations span the Lie algebra $\Der(\calH)$.

In the same vein,
for $x,y\in \calJ$, the linear map $d_{x,y}\colon \calJ\rightarrow \calJ$ defined by
\begin{equation}\label{eq:dxy}
d_{x,y}(z)=x(yz)-y(xz),
\end{equation}
is the inner derivation of $\calJ$ determined by the elements $x$ and
$y$. These derivations span the Lie algebra of derivations $\Der(\calJ)$

Given $\calH$ and $\calJ$ as before, consider the space
\begin{equation}\label{eq:TCJ}
\calT(\calH,\calJ)=\Der(\calH)\oplus \bigl(\calH_0\otimes \calJ_0\bigr)\oplus \Der(\calJ),
\end{equation}
with the
anticommutative multiplication $[.,.]$ specified by:
\begin{itemize}
\item $\Der(\calH)$ and $\Der(\calJ)$ are Lie subalgebras, and $[\Der(\calH),\Der(\calJ)]=0$,
\item $[D,a\otimes x]=D(a)\otimes x$, $[d,a\otimes x]=a\otimes
d(x)$,
\item $[a\otimes x,b\otimes y]=\trace(xy)D_{a,b}+\bigl([a,b]\otimes
x*y\bigr)+2t(ab)d_{x,y}$,
\end{itemize}
for all $D\in \Der(\calH)$, $d\in \Der(\calJ)$, $a,b\in \calH_0$, and $x,y\in \calJ_0$, where $x*y=xy-\frac{1}{3}\trace(xy)1$.

Looking at Equation \eqref{eq:TCJ} from the left, in case $\calH$ is the Cayley algebra $\bO$ (i.e., $\dim \calH=8$), then $\Der(\bO)$ is the simple Lie algebra of type $G_2$ and \eqref{eq:TCJ} gives a decomposition as in Equation \eqref{eq:reduced_isotypic}, thus proving that $\calT(\bO,\calJ)$ is graded by the root system of type $G_2$ with coordinate algebra $\calJ=\FF 1\oplus\calJ_0$.

Looking from the right, we obtain:
\begin{itemize}
\item If $\calJ$ is the Albert algebra $\bA$ (i.e.; $\calJ$ is the algebra of hermitian $3\times 3$-matrices over the Cayley algebra), then $\Der(\bA)$ is the simple Lie algebra of type $F_4$, and \eqref{eq:TCJ} proves that $\calT(\bO,\bA)$ is graded by the root system of type $F_4$ with coordinate algebra $\bO=\FF 1\oplus\bO_0$.
\item If $\calJ$ is the Jordan algebra $H_3(\bH)$, then $\Der(\calJ)$ is the simple Lie algebra of type $C_3$, and $\calT(\bO,\calJ)$ is graded by the root system of type $C_3$ with coordinate algebra $\bO$.
\item Also, if $\calJ$ is the Jordan algebra $\Mat_3(\bF)^+=H_3(\bK)$, then $\Der(\calJ)$ is the simple Lie algebra of type $A_2$, and then $\calT(\bO,\calJ)$ is graded by the root system of type $A_2$ with coordinate algebra $\bO$.
\end{itemize}

\begin{table}[h!]
\begin{tabular}{|c|c|c|c|c|c|c|c|}
\hline
Lie&Root&&&&&&coordinate\\
algebra&system&$\dim\calA$&$\dim\calB$&$\dim\calC$&$\dim\calD$&model& algebra\\
\hline\hline
$G_2$&$G_2$&$1$&$0$&--&$0$&&$\FF$\\
\hline\hline
$F_4$&$G_2$&$1$&$5$&--&$3$&$\calT(\bO,H_3(\bF))$&$H_3(\bF)$\\
\hline
$F_4$&$F_4$&$1$&$0$&--&$0$&&$\FF$\\
\hline\hline
$E_6$&$A_2$&$8$&--&--&$14$&$\calT(\bO,H_3(\bK))$&$\bO$\\
\hline
$E_6$&$BC_2$&$5$&$1$&$2$&$4$&&\\
\hline
$E_6$&$G_2$&$1$&$8$&--&$8$&$\calT(\bO,H_3(\bK))$&$\Mat_3(\bF)^+$\\
\hline
$E_6$&$F_4$&$1$&$1$&--&$0$&$\calT(\bK,\bA)$&$\bK$\\
\hline
$E_6$&$E_6$&$1$&$0$&--&$0$&&$\bF$\\
\hline\hline
$E_7$&$BC_2$&$7$&$1$&$8$&$9$&&\\
\hline
$E_7$&$G_2$&$1$&$14$&--&$21$&$\calT(\bO,H_3(\bH))$&$H_3(\bH)$\\
\hline
$E_7$&$C_3$&$1$&$7$&--&$14$&$\calT(\bO,H_3(\bH))$&$\bO$\\
\hline
$E_7$&$F_4$&$1$&$3$&--&$3$&$\calT(\bH,\bA)$&$\bH$\\
\hline
$E_7$&$E_7$&$1$&$0$&--&$0$&&$\FF$\\
\hline\hline
$E_8$&$BC_2$&$11$&$1$&$20$&$24$&&\\
\hline
$E_8$&$G_2$&$1$&$26$&--&$52$&$\calT(\bO,\bA)$&$\bA$\\
\hline
$E_8$&$F_4$&$1$&$7$&--&$14$&$\calT(\bO,\bA)$&$\bO$\\
\hline
$E_8$&$E_8$&$1$&$0$&--&$0$&&$\FF$\\
\hline
\end{tabular}
\caption{{\vrule height 18pt width 0pt}Gradings by root systems of rank $\geq 2$ of the exceptional simple Lie algebras.}\label{table:root_gradings}
\end{table}

\begin{theorem}\label{th:exceptional_rank>3}
The fine gradings, up to equivalence, of the exceptional simple Lie algebras whose universal group has free rank $\geq 3$ are the following:
\begin{itemize}
\item The Cartan gradings of $F_4$, $E_6$, $E_7$ and $E_8$. The universal group is $\bZ^r$ with $r$ the rank of the algebra.
\item The gradings of $E_r$, $r=6,7,8$ induced by gradings by the root system of type $F_4$. The universal  groups are $\bZ^4\times\bZ_2^{r-5}$, $r=6,7,8$, and the respective types are $(72,1,0,1)$, $(120,0,3,1)$ and $(216,0,0,8)$.
\item A grading of $E_7$ induced by a grading by the root system of type $C_3$. The universal group is $\bZ^3\times\bZ_2^3$ and its type is $(102,0,1,7)$.
\end{itemize}
\end{theorem}
\begin{proof}
The only gradings by root systems of rank $\geq 3$ in Table \ref{table:root_gradings} are the Cartan gradings, the gradings by the root system of type $F_4$ of $E_r$, $r=6,7,8$, and the grading by the root system of type $C_3$ of $E_7$. In the second case, the coordinate algebra is $\calH=\bK$, $\bH$ or $\bO$ respectively, and the only grading on these algebras with neutral component equal to $\FF 1$ are the gradings obtained by the Cayley-Dickson doubling process (see \cite{Elduque_octonions} or \cite{EK}), whose universal groups are $\bZ_2$, $\bZ_2^2$ and $\bZ_2^3$ respectively. The computation of the types is straightforward using the model $\calT(\calH,\bA)$. Finally, these gradings are fine as the neutral component is the Cartan subalgebra of the subalgebra $\Der(\bA)$ of type $F_4$, and the grading induced in this subalgebra is the Cartan grading, which is fine. Hence, if any of these gradings could be refined, the refinement would be attached to a grading by a root system of rank $\geq 4$, which is impossible.

Finally, the coordinate algebra for the grading by the root system of type $C_3$ of $E_7$ is $\bO$. The only grading of $\bO$ whose neutral component is $\bF 1$ is its $\bZ_2^3$-grading. The resulting grading by $\bZ^3\times\bZ_2^3$ of $E_7$ is fine and its type is easily computed using the model $\calT(\bO,H_3(\bH))$.
\end{proof}

We finish with the promised short proofs of the classification of fine gradings for $G_2$ and $F_4$. For $G_2$ it was proved independently in \cite{CristinaCandido_G2} and \cite{BahturinTvalavadze}, and for $F_4$ in \cite{CristinaCandido_F4} (see also \cite{Cristina_Noncompu} and \cite{EK}). The arguments here are very different in nature.

\begin{theorem}\label{th:G2}
Up to equivalence, the simple Lie algebra of type $G_2$ is endowed with two different fine gradings: the Cartan grading by $\bZ^2$, and a special grading by $\bZ_2^3$ in which the seven nonzero homogeneous spaces are all Cartan subalgebras.
\end{theorem}
\begin{proof}
Let $\Gamma\,\colon\,\calL=\bigoplus_{g\in G}\calL_g$ be a fine grading of the simple Lie algebra $\calL$ of type $G_2$, with $G$ its universal group. By Theorem \ref{th:root_grading} and Table \ref{table:root_gradings}, either $\Gamma$ is the Cartan grading, or the free rank of $G$ is one, or $G$ is a finite group.

If the free rank is one, then $\calL$ is graded by the root system $BC_1$ (this includes gradings by $A_1$) and hence $\calL$ is given by the Tits-Kantor-Koecher Lie algebra constructed from a structurable algebra and $\Gamma$ is obtained by combining the $\bZ$-grading given by the rank one root system, and a grading of the coordinate algebra as in Proposition \ref{pr:grading_coordinate_algebra}. A look at the possibilities in
\cite[\S 8]{Allison_Models_Isotropic} shows that the coordinate algebra is the structurable algebra $\fra=\Mat_2(\bF)$ with multiplication given by:
\[
\begin{pmatrix} \alpha&\beta\\ \gamma&\delta\end{pmatrix}
\begin{pmatrix} \alpha'&\beta'\\ \gamma'&\delta'\end{pmatrix}
=\begin{pmatrix}\alpha\alpha'+3\beta\gamma'&\alpha\beta'+\beta\delta'+2\gamma\gamma'\\
  \gamma\alpha'+\delta\gamma'+2\beta\beta' &\delta\delta'+3\gamma\beta'
  \end{pmatrix},
\]
and involution
\[
\overline{\begin{pmatrix} \alpha&\beta\\ \gamma&\delta\end{pmatrix}}=
\begin{pmatrix} \delta&\beta\\ \gamma&\alpha\end{pmatrix}.
\]
Consider the basis $\left\{
1=\left(\begin{smallmatrix}1&0\\ 0&1\end{smallmatrix}\right),
e=\left(\begin{smallmatrix}0&1\\ 0&0\end{smallmatrix}\right),
f=\left(\begin{smallmatrix}0&0\\ 1&0\end{smallmatrix}\right),
s=\left(\begin{smallmatrix}1&0\\ 0&-1\end{smallmatrix}\right)\right\}$ of $\fra$, so that $\calA=\espan{1,e,f}$ and $\calB=\bF s$. Since $s^2=1$ and $\calB_0=0$, $s\in \fra_g$ with $2g=0$. The subspace $\bF e+\bF f=\{x\in \fra: sx+xs=0\}$ is graded. But for any nonzero homogeneous element $\alpha e+\beta f$, the elements $s(\alpha e+\beta f)=-\alpha e+\beta f$, $(\alpha e+\beta f)^2=2(\beta^2 e+\alpha^2 f)+3\alpha\beta 1$ and $(\alpha e+\beta f)(-\alpha e+\beta f)=2(\beta^2 e+\alpha^2 f)+3\alpha\beta s$ are homogeneous too, and this forces the nonzero element $\beta^2 e+\alpha^2 f$ to be homogeneous of degree $0$ and $g$ at the same time, a contradiction.

Finally, if $G$ is finite, consider the finite quasitorus $Q$ of the algebraic group $\Aut(\calL)$ which is the image of the character group $\hat G$ (isomorphic to $G$). Since $\Gamma$ is fine, $Q$ is a maximal quasitorus. Also, since $\calL$ is of type $G_2$, $\Aut(\calL)$ is a connected and simply connected semisimple algebraic group. For any $\chi\in Q$, $\chi$ is semisimple and $\Aut(\calL)$ is connected and semisimple, so its centralizer $\Cent_{\Aut(\calL)}(\chi)$ is reductive \cite[Theorem 2.2]{Humphreys_Conjugacy}, and since $\Aut(\calL)$ is simply connected, $\Cent_{\Aut(\calL)}(\chi)$ is connected \cite[Theorem 2.11]{Humphreys_Conjugacy}. Hence the solvable radical coincides with the connected component of its center: $Z\bigl(\Cent_{\Aut(\calL)}(\chi)\bigr)^\circ$, and this is a torus \cite[Lemma 19.5]{Humphreys_LinearAlgebraicGroups}. But $Z\bigl(\Cent_{\Aut(\calL)}(\chi)\bigr)^\circ$ is contained in any maximal quasitorus of $\Cent_{\Aut(\calL)}(\chi)$, so it is contained in $Q$. But $Q$ is finite, so $Z\bigl(\Cent_{\Aut(\calL)}(\chi)\bigr)^\circ=0$, and hence $\Cent_{\Aut(\calL)}(\chi)\bigr)$ is semisimple, so that the subalgebra of $\calL$ of elements fixed by $\chi$ is semisimple.

The automorphisms of finite order of the simple Lie algebras are well-known (see \cite[Chapter 8]{Kac}). They are determined, up to conjugation, by a subset of nodes of the extended Dynkin diagram and some coefficients. Those automorphisms of finite order whose subalgebra of fixed elements is semisimple, correspond to the automorphisms attached to a single node. For $G_2$, the extended Dynkin diagram (with coefficients) is:
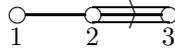
\[
\begin{picture}(30,10)
	% Circles
	\put(0,5){\circle{2}}
	\put(10,5){\circle{2}}
	\put(20,5){\circle{2}}
	
	% Lines
	\put(1,5){\line(1,0){8}} %single
    \put(11,5){\line(1,0){8}}
	\multiput(10,4)(0,2){2}{\line(1,0){10}} %triple arrow
	
	% Labels
	\put(-1,1){\small 1}
	\put(9,1){\small 2}
	\put(19,1){\small 3}
	
	% Arrows
	\put(14.5,4.2){$\rangle$}
\end{picture}
\]

Therefore, the order of a nontrivial finite order automorphism of $\calL$ whose subalgebra of fixed elements is semisimple is restricted to $2$ or $3$. Thus $Q$ is a maximal nontoral elementary $p$-subgroup of $\Aut(\calL)$, with $p=2$ or $3$. According to \cite{Griess}, there is just one possibility, up to conjugation, where $Q$, and hence $G$, isomorphic to $\bZ_2^3$.
\end{proof}

\begin{theorem}\label{th:F4}
Up to equivalence, the simple Lie algebra $\calL$ of type $F_4$ is endowed with four different fine gradings, whose universal groups and types are as follows:
\begin{itemize}
\item the Cartan grading by $\bZ^4$, of type $(48,0,0,1)$,
\item a grading by $\bZ\times\bZ_2^3$, of type $(31,0,7)$,
\item a grading by $\bZ_2^5$ of type $(24,0,0,7)$, and
\item a grading by $\bZ_3^3$ of type $(0,26)$, such that for any $0\ne \alpha\in\bZ_3^3$, $\calL_\alpha\oplus\calL_{-\alpha}$ is a Cartan subalgebra of $\calL$.
\end{itemize}
\end{theorem}
\begin{proof}
Let $\Gamma\,\colon\,\calL=\bigoplus_{g\in G}\calL_g$ be a fine grading of the simple Lie algebra $\calL$ of type $F_4$, with $G$ its universal group. By Theorem \ref{th:root_grading} and Table \ref{table:root_gradings}, either $\Gamma$ is the Cartan grading, or the free rank of $G$ is two and $\Gamma$ is associated to a grading by the root system of type $G_2$, or the free rank of $G$ is one, or $G$ is a finite group.

If the free rank of $G$ is $2$, the the coordinate algebra (see Table 1) is the Jordan algebra $H_3(\FF)$ of symmetric $3\times 3$-matrices. But the results of \cite{BSZ} show that the neutral component of any grading on $H_3(\FF)$ by any group has dimension at least $3$, and this contradicts Proposition \ref{pr:grading_coordinate_algebra}.

If the free rank of $G$ is one, then $\calL$ is graded by the root system $BC_1$ (this includes gradings by $A_1$) and hence $\calL$ is given by the Tits-Kantor-Koecher Lie algebra constructed from a structurable algebra and $\Gamma$ is obtained by combining the $\bZ$-grading given by the rank one root system, and a grading of the coordinate algebra as in Proposition \ref{pr:grading_coordinate_algebra}. A look at the possibilities in \cite[\S 8]{Allison_Models_Isotropic} shows that the coordinate algebra is either the Cayley algebra $\bO$, with its standard involution, or a structurable algebra defined on the vector space of matrices $\left(\begin{smallmatrix} \alpha&a\\ b&\beta\end{smallmatrix}\right)$, with $\alpha,\beta\in \FF$ and $a,b\in H_3(\FF)$.

For the Cayley algebra, there is a unique grading, up to equivalence, whose neutral component is $\FF 1$, with universal group $\bZ_2^3$, thus obtaining the grading by $\bZ\times\bZ_2^3$.

In the second case, the coordinate algebra $\fra=\cA\oplus\cB$ has dimension $14$, with $\dim\cB=1$. Moreover, $\cB=\FF s$ for an element $s$ with $s^2=1$, and hence $\cB=\cB_g$ for an element $0\ne g\in \tor(G)$ with $2g=0$. The Lie algebra $\calL$ decomposes as in \eqref{eq:BC1_isotypic}, and the neutral component of the associated grading by the root system of type $BC_1$ decomposes as:
\begin{equation}\label{eq:F4_BC1}
\calL(0)\simeq \cA\oplus\cB\oplus\cD\simeq \fra\oplus\calD.
\end{equation}
This is a reductive Lie algebra with one-dimensional center (corresponding to $\FF 1$) and derived subalgebra simple of type $C_3$. (Actually $\calL(0)$ is isomorphic to the structure Lie algebra $\Str(\fra,-)$, see \cite[\S 1]{Allison_Models_Isotropic}.) On the other hand, $\calD$ is isomorphic to the Lie algebra of derivations of $\fra$, which is simple of type $A_2$. The results in \cite{E10} show that the simple Lie algebra of type $C_3$ is endowed with a unique grading with trivial neutral component, with universal group $\bZ_2^4$ and type $(12,0,3)$. On the other hand, the simple Lie algebra of type $A_2$ has a unique grading, up to equivalence, with trivial neutral component and whose universal group is $2$-elementary. Its type is $(6,1)$. It turns out that $\tor(G)$ is $2$-elementary and that at least two of the three homogeneous components of $[\calL(0),\calL(0)]$ of dimension $3$ intersect the graded subspace $\fra$ in \eqref{eq:F4_BC1} with dimension $\geq 2$. We conclude that there is an element $0\ne h\in\tor(G)$ such that $\dim \cA_h\geq 2$, and $h\ne g$ (recall $\cB=\cB_g$). Since $\fra$ is a simple structurable algebra, the form $\langle x,y\rangle=\trace(L_{x\bar y+y\bar x})$ is nondegenerate \cite{AllisonSchafer}. But $\langle \fra_{g_1},\fra_{g_2}\rangle=0$ unless $g_1+g_2=0$. Therefore, the restriction of this form to $\cA_h$ is nondegenerate. Now, for any two elements $x,y\in\cA_h$, $xy\in \fra_0=\FF1$, so $xy=\alpha 1=\overline{xy}=\bar y\bar x=yx$ and $\langle x,y\rangle=\trace(L_{2\alpha 1})=2\alpha\dim\fra$. We may then find elements $x,y\in \cA_h$ with $x^2=0=y^2$ and $xy=1$. But then the derivation $D_{x,y}$ in \eqref{eq:Duv} satisfies $D_{x,y}(x)\ne 0$, so $0\ne D_{x,y}\in\calD_0$, a contradiction with $\calD_0=0$.

We are left with the case in which $G$ is finite. As in the proof of Theorem \ref{th:G2}, we consider the extended Dynkin diagram:

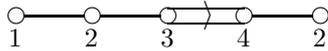
\[
\begin{picture}(50,10)
	% Circles
	\put(0,5){\circle{2}}
	\put(10,5){\circle{2}}
	\put(20,5){\circle{2}}
    \put(30,5){\circle{2}}
    \put(40,5){\circle{2}}
	
	% Lines
    \put(1,5){\line(1,0){8}}
    \put(11,5){\line(1,0){8}}
    \put(31,5){\line(1,0){8}} %single arrows
	\multiput(20,4)(0,2){2}{\line(1,0){10}} %double arrow
	
	% Labels
	\put(-1,1){\small 1}
	\put(9,1){\small 2}
	\put(19,1){\small 3}
    \put(29,1){\small 4}
    \put(39,1){\small 2}
	
	% Arrows
	\put(24.5,4.2){$\rangle$}
\end{picture}
\]
and check that either $G$ is an elementary $2$-group or $3$-group, or the associated quasitorus $Q$($\simeq \hat G$) contains an automorphism $\chi$ of order $4$.

In the latter case, the subalgebra of elements fixed by $\chi$ is isomorphic to $\frsl(V)\oplus\frsl(W)$ with $\dim V=4$, $\dim W=2$ (see \cite[Chapter 8]{Kac}) and the other eigenspaces of $\chi$ are, as modules for $\frsl(V)\oplus\frsl(W)$, isomorphic to $V\otimes W$, $\wedge^2V\otimes S^2W$ and $V^*\otimes W$, with respective eigenvalues $\sqrt{-1},-1,-\sqrt{-1}$. The action of any automorphism in the connected subgroup $\Cent_{\Aut(\calL)}(\chi)$ is determined by its restriction to $V\otimes W$. It is not difficult to check now that $\Cent_{\Aut(\calL)}(\chi)$ is isomorphic to $SL(V)\times SL(W)/\langle \pm(I_V,I_W)\rangle$ ($I_X$ denotes the identity map on the vector space $X$). For $f\in SL(V)$ and $g\in SL(W)$, denote by $\psi_{f,g}$ the automorphisms of $\calL$ such that $\psi_{(f,g)}\vert_{V\otimes W}=f\otimes g$. Moreover, $\Gamma$ induces gradings on $\frsl(V)$ and on $\frsl(W)$ with trivial neutral components, induced by the projections $\pi_V\colon SL(V)\times SL(W)/\{\pm(I_V,I_W)\}\rightarrow PSL(V)=SL(V)/\langle\sqrt{-1}I_V\rangle$ (contained, up to isomorphism, in $\Aut(\frsl(V))$), and $\pi_W\colon SL(V)\times SL(W)/\{\pm(I_V,I_W)\}\rightarrow PSL(W)=SL(W)/\{\pm I_W\}$.

There is \cite{E10}, up to equivalence, only one one possibility for such grading on $\frsl(W)$, where $\pi_W(Q)=\langle \bar g_1,\bar g_2\rangle$, with $g_1,g_2\in SL(W)$ or order $2$, $g_1g_2=-g_2g_1$ and $\bar g_i$ denotes the class of $g_i$ in $PSL(W)$. With $g_0=I_W$, $g_3=g_1g_2$, and $Q_V^i=\{f\in SL(V): \psi_{f,g_i}\in Q\}$ we have $Q=\cup_{i=0}^3\psi_{Q_V^i,g_i}$. Since $Q$ is abelian, $\psi_{f,g}\psi_{f,g'}=\psi_{f',g'}\psi_{f,g}$, and it follows from $g_1g_2=-g_2g_1$, that the elements of $Q_V^i$ anticommute with the elements of $Q_V^j$ for $1\leq i\ne j\leq 3$, and that the elements of $Q_V^0$ commute with the elements in any $Q_V^i$.

Now, there are \cite{E10}, up to equivalence, only two possibilities of gradings on $\frsl(V)$ whose associated quasitorus is contained in $PSL(V)$ and whose neutral component is trivial. In the first of this possibilities, $\pi_V(Q)=\langle \bar f_1,\bar f_2\rangle$ with $f_1f_2=\sqrt{-1}f_2f_1$ but since  any two elements of $\pi_V(Q)$ must either commute or anticommute by the above, this is not possible. In the other possibility $\pi_V(Q)=\langle \bar f_1,\bar f_2,\bar f_1',\bar f_2'\rangle$, with $f_1,f_2,f_1',f_2'$ order two elements of $SL(V)$ such that $f_1f_2=-f_2f_1$, $f_1'f_2'=-f_2'f_1'$ and $f_if_j'=f_j'f_i$ for any $i,j=1,2$. We may assume, scaling the elements if necessary, that $f_1\in Q_V^1$ and $f_2\in Q_V^2$. But then, up to scalars, $f_1'$ and $f_2'$ must belong to $Q_V^0$, since they commute with both $f_1$ and $f_2$. This is a contradiction, since $f_1'$ and $f_2'$ anticommute.

We conclude that, if $G$ is finite, the maximal quasitorus cannot contain automorphisms of order $4$, and hence $G$ is an elementary $2$ or $3$-group, and the results in \cite{Griess} prove that either $G\cong \bZ_2^5$ or $G\cong \bZ_3^3$. The description of the gradings (with the exception of the Cartan grading) and their types appear, for instance, in \cite{Elduque_SymCom}.
\end{proof}

%-----------------------------------------

%------------------------------------------


\bigskip
\begin{thebibliography}{ABG02}

\bibitem[All78]{Allison78}
B.N.~Allison,  \emph{A class of nonassociative algebras with involution containing the class of {J}ordan algebras},
Math. Ann. \textbf{237} (1978), no.~2, 133--156.

\bibitem[All79]{Allison_Models_Isotropic}
B.N.~Allison, \emph{Models of isotropic simple Lie algebras},
Comm. Algebra \textbf{7} (1979), no.~17, 1835--1875.

\bibitem[ABG02]{ABG}
B.~Allison, G.~Benkart, and Y.~Gao,
     \emph{Lie algebras graded by the root systems {$BC_r,\ r\ge2$}},
Mem. Amer. Math. Soc. \textbf{751} (2002), x+158 pp.

\bibitem[AS89]{AllisonSchafer}
B.N.~Allison, R.D.~Schafer, \emph{Trace form for structurable algebras},
J.~Algebra \textbf{121} (1989), no.~1, 68--80.

\bibitem[BSZ05]{BSZ}
Y.A.~Bahturin, I.P.~Shestakov, M.~Zaicev, \emph{Gradings on simple Jordan and Lie algebras}, J.~Algebra \textbf{283} (2005), no.~2, 849--868.

\bibitem[BT09]{BahturinTvalavadze}
Y.A.~Bahturin and M.~Tvalavadze, \emph{Group gradings on $G_2$} Comm. Algebra \textbf{37} (2009), no.~3, 885--893.

\bibitem[BS03]{Benkart_Smirnov}
G.~Benkart and O.~Smirnov, \emph{Lie algebras graded by the root system $BC_1$},
J.~Lie Theory \textbf{13} (2003), no.~1, 91--132.

\bibitem[BZ96]{BZ}
G.~Benkart and E.~Zelmanov, \emph{Lie algebras graded by finite root systems and intersection matrix algebras}, Invent. Math. \textbf{126} (1996), 1--45.

\bibitem[BM92]{BermanMoody}
S.~Berman and R.V.~Moody, \emph{Lie algebras graded by finite root systems and the intersection matrix algebras of Slodowy}, Invent. Math. \textbf{108} (1992), 323--347.

\bibitem[Bou98]{Bou98}
N.~Bourbaki, \emph{Lie groups and {L}ie algebras. {C}hapters 1--3},
  Elements of Mathematics (Berlin), Springer-Verlag, Berlin, 1998,
  Translated from the French, Reprint of the 1989 English translation.

\bibitem[Dra12]{Cristina_Noncompu}
C.~Draper, \emph{A non-computational approach to the gradings on $\mathfrak{f}_4$}, Rev. Mat. Iberoam. \textbf{28} (2012), no.~1, 273--296.

\bibitem[DM06]{CristinaCandido_G2}
C.~Draper and C.~Mart\'{\i}n-Gonz\'alez, \emph{Gradings on $\frg_2$}, Linear Algebra Appl. \textbf{418} (2006), no.~1, 85--111.

\bibitem[DM09]{CristinaCandido_F4}
C.~Draper and C.~Mart\'{\i}n-Gonz\'alez, \emph{Gradings on the Albert algebra and on $\mathfrak{f}_4$}, Rev. Mat. Iberoam. \textbf{25} (2009), no.~3, 841--908.

\bibitem[DV]{CristinaViru}
C.~Draper and A.~Viruel, \emph{Fine gradings on $\mathfrak{e}_6$}, arXiv:1207.6690v1 [math.RA].

\bibitem[Dyn52]{Dynkin}
E.B.~Dynkin, \emph{Semisimple subalgebras of semisimple Lie algebras} (Russian) Mat. Sbornik N.S. \textbf{30(72)} (1952), no.~2, 349--462. Amer. Math. Translations Ser. 2 \textbf{6} (1957), 111--244.

\bibitem[Eld98]{Elduque_octonions}
A.~Elduque, \emph{Gradings on octonions}, J.~Algebra \textbf{207} (1998), no.~1, 342--354.

\bibitem[Eld09]{Elduque_SymCom}
Alberto Elduque, \emph{Gradings on symmetric composition algebras},
   J.~Algebra \textbf{322} (2009), no.~10, 3542--3579.

\bibitem[Eld10]{E10}
A.~Elduque, \emph{Fine gradings on simple classical {L}ie algebras},
    J.~Algebra \textbf{324} (2010), no.~12, 3532--3571.

\bibitem[EK12]{EK}
A.~Elduque and M.~Kochetov, \emph{Gradings on the exceptional Lie algebras $F_4$ and $G_2$ revisited}, Rev. Mat. Iberoam. \textbf{28} (2012), no.~3, 773--813.

\bibitem[EO08]{EO_Tits}
A.~Elduque and S.~Okubo, \emph{$S_4$-symmetry on the Tits construction of exceptional Lie algebras and superalgebras}, Publ. Mat. \textbf{52} (2008), no.~2, 315--346.

\bibitem[Gri91]{Griess}
R.L.~Griess Jr., \emph{Elementary abelian $p$-subgroups of algebraic groups},
Geom. Dedicata \textbf{39} (1991), no.~3, 253--305.

\bibitem[Hes82]{Hesselink}
W.H.~Hesselink, \emph{Special and pure gradings of {L}ie algebras},
   Math.~Z. \textbf{179} (1982), no.~1, 135--149.

\bibitem[Hum75]{Humphreys_LinearAlgebraicGroups}
J.E.~Humphreys, \emph{Linear algebraic groups}, Graduate Texts in Mathematics \textbf{21}. Springer-Verlag, New York-Heidelberg, 1975, xiv+247 pp.

\bibitem[Hum95]{Humphreys_Conjugacy}
J.E.~Humphreys, \emph{Conjugacy classes in semisimple algebraic groups}, Mathematical Surveys and Monographs, \textbf{43}. American Mathematical Society, Providence, RI, 1995, xviii+196 pp.

\bibitem[Jac79]{Jacobson}
N.~Jacobson, \emph{Lie algebras}, Dover Publications Inc., New York, 1979,
  Republication of the 1962 original.

\bibitem[Kac90]{Kac}
V.G.~Kac, \emph{Infinite-dimensional Lie algebras}, Third edition. Cambridge University Press, Cambridge, 1990. xxii+400 pp.

\bibitem[PZ89]{PZ}
J.~Patera and H.~Zassenhaus, \emph{On {L}ie gradings. {I}},
    Linear Algebra Appl. \textbf{112} (1989), 87--159.

\bibitem[Tit66]{Tits66}
J.~Tits, \emph{Alg\`ebres alternatives, alg\`ebres de {J}ordan et
alg\`ebres de
  {L}ie exceptionnelles. {I}. {C}onstruction}, Nederl. Akad. Wetensch. Proc.
  Ser. A \textbf{69} = Indag. Math. \textbf{28} (1966), 223--237.

\end{thebibliography}
\end{document}